\pdfoutput=1
\documentclass[10pt,letterpaper]{amsart}

\usepackage[utf8]{inputenc}
\usepackage{lmodern}
\usepackage{mathtools,amsfonts,amsthm,mathrsfs,amssymb}
\usepackage{bbm}
\usepackage[T1]{fontenc}

\usepackage{enumitem}
\usepackage[dvipsnames]{xcolor}
\usepackage{subfig}

\newenvironment{axioms}
{\begin{enumerate}[label=A\arabic*.,ref=A\arabic*]}
{\end{enumerate}}

\newcounter{condition}

\makeatletter
\newcommand{\axitem}[1][]{%
  \if\relax\detokenize{#1}\relax
    \item
  \else
    \item[#1]%
    \def\thecondition{#1}%
    \refstepcounter{condition}%
  \fi
}
\makeatother

\usepackage[margin=1in]{geometry}

\usepackage{svg}

\usepackage{textcomp}
\mathtoolsset{showonlyrefs}

\PassOptionsToPackage{pdfusetitle,%
    colorlinks=true,%
    pdfstartview=FitV,%
    linkcolor=blue,%
    citecolor=Green,%
    urlcolor=WildStrawberry,%
    linktoc=page%
    }{hyperref}
\usepackage{bookmark}
\usepackage[capitalise,noabbrev]{cleveref}
\usepackage{microtype}

\newtheorem{theorem}{Theorem}
\newtheorem{lemma}[theorem]{Lemma}
\newtheorem{proposition}[theorem]{Proposition}
\newtheorem{corollary}[theorem]{Corollary}

\theoremstyle{definition}
\newtheorem{definition}[theorem]{Definition}

\Crefname{theorem}{Theorem}{Theorems}
\Crefname{lemma}{Lemma}{Lemmas}
\Crefname{proposition}{Proposition}{Propositions}
\Crefname{corollary}{Corollary}{Corollaries}
\Crefname{conjecture}{Conjecture}{Conjectures}
\Crefname{remark}{Remark}{Remarks}
\Crefname{definition}{Definition}{Definitions}
\Crefname{figure}{Figure}{Figures}
\Crefname{enumi}{Item}{Items}
\Crefname{condition}{Condition}{Conditions}

\newcommand*{\RR}{\mathbb{R}}

\newcommand*{\TT}{\mathbb{T}}

\newcommand*{\cI}{\mathcal{I}}

\renewcommand{\epsilon}{\varepsilon}

\renewcommand{\subset}{\subseteq}

\newlist{enumerata}{enumerate}{1}
\setlist[enumerata]{label=\upshape{(\alph*)}}
\setlist[enumerate]{label=\upshape{(\roman*)}}

\title{First-order phase transitions in the hard-core model}
\date{\today}

\author[E.\ Davies]{Ewan Davies}
\address{Department of Computer Science, Colorado State University, Fort Collins, USA}
\email{research@ewandavies.org}

\author[J.S.\ Sandhu]{Juspreet Singh Sandhu}
\address{Department of Computer Science, Colorado State University, Fort Collins, USA}
\email{js.sandhu@colostate.edu}

\author[B.\ Tan]{Brian Tan}
\address{Department of Computer Science, Colorado State University, Fort Collins, USA}
\email{Brian.Tan@colostate.edu}

\begin{document}
\begin{abstract}
We give a rigorous proof that the hard-core model on a natural family of infinite graphs exhibits a first-order phase transition. 
\end{abstract}

\maketitle

\section{Introduction}

The hard-core model is a natural measure on the independent sets of a graph, with applications in physics as a model of a gas~\cite{vS94} and in models of communication networks~\cite{Kel85,Kel91}.
The hard-core model provides a simple example of a system with hard constraints:
given a graph $G=(V,E)$, we can think of the state space as the set of assignments of spins in $\{0,1\}$ to the vertices of $G$ such that no two adjacent vertices receive spin $1$. Equivalently, the state space is the set of independent sets of $G$.

Given a fixed finite graph $G$, there is an infinite family of hard-core models on $G$ parametrized by a vector $\lambda \in (0,\infty)^V$ (we exclude the possibility of coordinates equal to zero for mild technical convenience). That is, each vertex can have a parameter $\lambda_v$ which loosely controls how likely $v$ is to receive spin $1$.
In this work, we focus on the univariate case where each vertex has the same parameter $\lambda>0$, which we call the fugacity.
In this case, a finite graph has exactly one hard-core model associated with a given fugacity. 
One reason the hard-core model is well-studied in statistical physics is that for a fixed \emph{infinite} graph there can be multiple hard-core models with the same parameter $\lambda$. 
This leads us to the notion of a phase transition, our main interest. Before stating what we mean by this, we give a disclaimer by quoting Brightwell and Winkler~\cite{BW02}.
\begin{quote}
    There is no uniformity even among statistical physicists regarding the definition of a phase transition.
\end{quote}

\begin{definition}\label[definition]{def:phasetrans}
    If, for a given infinite graph $G$ and fugacity $\lambda>0$, there is only one hard-core measure on $G$ at fugacity $\lambda$, then we say $G$ has \emph{uniqueness} at $\lambda$. If there are multiple such measures, we say that $G$ has \emph{phase coexistence} at $\lambda$.
    If there exists some $\lambda^*=\lambda^*(G)$ such that $G$ has uniqueness for each $\lambda < \lambda^*$ and $G$ has phase coexistence for each $\lambda>\lambda^*$, then we say that a \emph{phase transition} occurs at $\lambda^*$, or more specifically that $\lambda^*$ is the \emph{Gibbs uniqueness threshold} for $G$.
\end{definition}

Leaving aside the question of how hard-core models are defined on infinite graphs, which is briefly discussed below, we point out that the simple scenario described by our definition may not exist. 
General considerations for Gibbs measures (see e.g.~\cite{Geo11}) mean that at least one hard-core measure must exist at a given $\lambda$, but even in cases where, as $\lambda$ varies, we can characterize precisely the regions of uniqueness and phase coexistence, we may not have a single phase transition as defined above.
Indeed, Brightwell, H\"aggstr\"om, and Winkler proved~\cite{BHW99} that there are infinite trees with values $0<\lambda_1<\lambda_2$ such that there is uniqueness for $\lambda\in[0,\lambda_1]\cup[\lambda_2,\infty)$ and phase coexistence for $\lambda\in(\lambda_1,\lambda_2)$.
It simplifies the exposition to restrict our attention to the strong kind of monotonicity implicitly assumed in the above definition, without significantly limiting our exploration of phase transitions and related phenomena. 

If $\mu$ is a hard-core model on a graph and $v$ is a vertex of the graph, then we write $\mu(v)$ for the marginal of $v$, which is the probability that $v\in I$ when $I\sim\mu$.
We are interested in the order of a phase transition, which we define as follows (cf.~\cite{GMRT11} for a similar definition in the case of vertex-transitive graphs). 

\begin{definition}\label[definition]{def:orders}
    Suppose that $G$ is an infinite graph such that a phase transition occurs at $\lambda^*>0$. 
    For $\lambda>0$ and a vertex $u$ of $G$, let $D_{G,u}(\lambda)$ be defined as 
    \[ D_{G,u}(\lambda) = \sup_{\mu,\nu}\left| \mu(u) - \nu(u) \right|, \]
    where the supremum is over all pairs of hard-core measures $\mu,\nu$ on $G$ at fugacity $\lambda$.
    Suppose also that for each vertex $u$, the one-sided limit from above $\ell_{G,u}(\lambda^*) = \lim_{\lambda\downarrow \lambda^*} D_{G,u}(\lambda)$ exists.
    
    If there exists a vertex $u\in V(G)$ such that $\ell_{G,u}(\lambda^*)>0$, then we say that $\lambda^*$ is a \emph{first-order} phase transition.
    If, on the other hand, $\ell_{G,u}(\lambda^*)=0$ for all $u\in V(G)$, then we say that $\lambda^*$ is a \emph{second-order} phase transition.
\end{definition}

One intuition behind classifying phase transitions as first- or second-order is the continuity of marginals in the model as one thinks of increasing $\lambda$ across the phase transition.
This is not a rigorous notion, however, as there is no way of knowing or determining which of the multiple hard-core measures one should take at and beyond the threshold.  
For a second-order phase transition, the intuition is that whichever measure one chooses as $\lambda$ crosses the threshold $\lambda^*$, the marginals of every vertex are continuous.
But for a first-order phase transition it is possible to choose a trajectory through the measures such that one can observe a discontinuity in a marginal. 
It is possible to make this intuition precise subject to a structural condition on the graph. For a so-called amenable graph (one where boundaries do not grow too fast; see e.g.~\cite{BLPS99}), one can define a free energy density for all $\lambda$ and study phase transitions and their orders from the perspective of non-analytic points and their order.
While this is an interesting perspective, we study non-amenable graphs in this work and thus cannot work with free energies in this way.

Perhaps the most well-known phase transition in the context of the hard-core model is on the infinite $\Delta$-regular tree~\cite{Kel85,Spi75}. There, for $\Delta\ge 3$, the transition occurs at
\begin{equation}
    \lambda_c(\Delta) := \frac{(\Delta-1)^{\Delta-1}}{(\Delta-2)^\Delta}
\end{equation}
and is second-order (see e.g.~\cite{GMRT11}).
This phase transition is interesting in computer science as it marks the threshold at which the problem of approximately sampling from the hard-core model on a finite graph of maximum degree $\Delta$ goes from polynomial-time computable to RP-hard~\cite{GSV16,Sly10,SS12,Wei06}.
It also corresponds to a \#BIS-hardness threshold in bipartite graphs~\cite{CGG+14}.

Our main result is that for integers $\Delta>\delta\ge 3$, the infinite $(\Delta,\delta)$-biregular tree, a countably infinite tree in which every edge lies between a vertex of degree $\Delta$ and a vertex of degree $\delta$, undergoes a first-order phase transition. The precise location of the threshold is the solution of an explicit equation involving $\Delta$ and $\delta$.

\begin{theorem}
    For all integers $\Delta>\delta\ge 3$, there is a first-order phase transition in the hard-core model on the infinite $(\Delta,\delta)$-biregular tree $\TT_{\Delta,\delta}$.
\end{theorem}

\section{Preliminaries}

We begin with a brief overview of the background necessary to understand the hard-core model on finite and infinite graphs.

\begin{definition}\label[definition]{def:pf}
    For a graph $G=(V,E)$, let $\cI(G)=\{I\subset V : \forall u,v\in I, uv\notin E\}$ be the set of independent sets in $G$ (including the empty set). If $G$ is finite, then for a formal parameter $\lambda$, the \emph{partition function} $Z_G(\lambda)$ is defined by
    \[ Z_G(\lambda) = \sum_{I\in \cI(G)}\lambda^{|I|}.\]
\end{definition}

\subsection{The hard-core model on finite graphs}

\begin{definition}\label[definition]{def:finitebasics}
    For a finite graph $G$ and $\lambda\ge 0$, the \emph{hard-core model on $G$ at fugacity $\lambda$} is the probability measure $\mu_{G,\lambda}$ on $\cI(G)$ such that
    \[ \mu_{G,\lambda}(I) = \frac{\lambda^{|I|}}{Z_G(\lambda)}. \]
    
    When the graph and fugacity are clear from context, we consider $I\sim\mu_{G,\lambda}$, identify a vertex $u$ with the indicator random variable for the event $u\in I$, and write $\overline{u}$ for the indicator of $u\notin I$. Then $\mu_{G,\lambda}(u)$ is the marginal of $u$ and $\mu_{G,\lambda}(\overline{u})=1-\mu_{G,\lambda}(u)$.
    
    In keeping with statistical physics terminology, we refer to the state of a vertex $u$ being in or out of an independent set as the \emph{spin} of $u$. One can think of an independent set $I\in\cI(G)$ as an assignment of spins in $\{0,1\}$ to the vertices of $G$ such that vertices in $I$ receive spin $1$ and vertices in the complement of $I$ receive spin $0$.
    For convenience when working with this formulation, we write $\Omega(G)$ for the set of functions $\sigma:V(G)\to\{0,1\}$ such that $\sigma^{-1}(1)\in\cI(G)$.
\end{definition}

\subsection{Infinite graphs}

In a graph $G=(V,E)$, for a subset $U\subset V$, we define the boundary $\partial U$ as
\[ \partial U = \{v\in V\setminus U : uv\in E \text{ for some } u\in U\}. \]

The hard-core model on a finite graph $G=(V,E)$, as an example of a spin system or a Gibbs measure, enjoys something called the \emph{spatial Markov property}. 
This is the property that for every subset $U\subset V$, conditioned on the spins of the boundary $\partial U$, the spins inside $U$ and outside $U$ are independent.
One can prove this for finite graphs in an elementary fashion from the definition of the hard-core model given above; see e.g.~\cite{DK25}.
The way we extend the definition of the model to infinite graphs is to require that for any finite subset $U$, once the spins are known outside of $U$, the model inside $U$ matches the finite hard-core model conditioned on the spins of the boundary $\partial U$. This is the Dobrushin--Lanford--Ruelle (DLR) approach~\cite{Dob68a,LR69}; see e.g.~\cite[Chap.~6]{FV17} for the details.

\subsection{Dynamical systems}

The study of phase transitions on trees is intimately related to dynamical systems arising from the iteration of rational maps.
We use standard terminology; see for instance~\cite{ASY97,Bac22}.

Let $A\subset \RR\cup\{\infty\}$ and let $f:A\to A$ be a differentiable function. 
We write $f^{\circ n}$ for the $n$-th iterate of $f$ in the sense that $f^{\circ 1}=f$ and $f^{\circ n} = f\circ f^{\circ n-1}$.

A point $x\in A$ such that $x=f(x)$ is called a fixed point of $f$. 
A fixed point $x$ is called \emph{attracting} if $|f'(x)|<1$, and the \emph{basin of attraction} is the largest subset $B\subset A$ such that, for all $y\in B$, $f^{\circ n}(y)\to x$ as $n\to\infty$.
A fixed point $x$ is called \emph{repelling} if $|f'(x)|>1$. 

A point $x\in A$ is called \emph{periodic} if there exists a finite $k\ge 1$ such that $f^{\circ k}(x)=x$, and the least such $k$ is called the \emph{period} of $x$.
A \emph{periodic orbit} is the (finite) set $\{f^{\circ n}(x) : n\ge 1\}$ for a periodic point $x$. 
The periodic orbit of the point $x\in A$ of period $k$ is called \emph{attracting} if $x$ is an attracting fixed point of $f^{\circ k}$ (and similarly for repelling). The basin of attraction of a periodic orbit is defined analogously.

Given a fixed fugacity $\lambda>0$, the occupation ratio $R_{G,v}$ of a vertex $v$ in a finite graph $G$ is given by $\mu_{G,\lambda}(v)/\mu_{G,\lambda}(\overline{v})$. Then
\[ \mu_{G,\lambda}(v) = \frac{R_{G,v}}{1+R_{G,v}}, \]
and we freely move between coordinates for the marginals and the ratios for convenience. 
In an infinite graph, the same definition applies, but one must specify the hard-core measure used to compute the ratio when multiple measures exist.
The following proposition demonstrates that ratios can be computed efficiently on trees by a recursion. Given a tree, designate the desired vertex as the root. Since a single vertex has ratio $\lambda$ in the one-vertex graph containing it, one starts the recursion with this value at the leaves and uses the proposition to compute the ratio at the root of the entire tree.

\begin{proposition}\label[proposition]{prop:rationbrs}
   For any finite tree $T$ rooted at $v \in V(T)$, the occupation ratio under $\mu_{T,\lambda}$ satisfies
   \[R_{T, v} = \frac{\lambda}{\prod_{i = 1}^d \left(1+ R_{T_i, v_i}\right)},\]
   where $(v_1, \cdots, v_d)$ are neighbors of $v$, and $T_i$ is the subtree rooted at $v_i$. 
\end{proposition}

The proposition is a well-known fact that has been used extensively in this field (see e.g.~\cite[Lem.~2.2]{PetReg19} for a proof of a generalization). In the special case of $d$-ary trees truncated to some fixed depth, every vertex at a given depth will have the same ratio. This lets us study the rational map $f_{d,\lambda}(r) = \lambda / (1+r)^d$, where iterating one step of the dynamical system is equivalent to increasing the depth of the truncated tree by one. Analogous reasoning applies to $(a,b)$-ary trees, where we can instead study $f_{a,\lambda} \circ f_{b,\lambda}$.

\section{Regular trees}

In this section, we give a short proof of a well-known result usually attributed to Kelly~\cite{Kel85,Kel91} (see also Spitzer~\cite{Spi75} and related early work on phase transitions on trees), though interestingly the phase transition demonstrated below arose in Kelly's work on loss networks: graph-based models of telephone communication networks that generalize the hard-core model.
We add some details on the nature of the phase transition (which are also well-known observations; see e.g.~\cite{GMRT11}).
As we will study a broader class of trees than the commonly considered $\Delta$-regular and $d$-ary trees, we state and prove slightly more general versions of standard facts from various works, e.g.~\cite{GMRT11,Spi75,Zac83,Zac85}.

\begin{theorem}\label[theorem]{thm:regtree}
    For $\Delta\ge 3$, the value
    \[ \lambda_c(\Delta)=\frac{(\Delta-1)^{\Delta-1}}{(\Delta-2)^\Delta} \]
    is the Gibbs uniqueness threshold for the hard-core model on the infinite $\Delta$-regular tree $\TT_\Delta$ in the sense that
    \begin{enumerate}
        \item for $\lambda< \lambda_c(\Delta)$, $\TT_\Delta$ has uniqueness at $\lambda$, and
        \item for $\lambda> \lambda_c(\Delta)$, $\TT_\Delta$ has phase coexistence at $\lambda$.
    \end{enumerate}
    Further, the phase transition on $\TT_\Delta$ at $\lambda_c(\Delta)$ is second-order in the sense of \Cref{def:orders}.
\end{theorem}

Let $G=(A\cup B,E)$ be a bipartite graph.
Given $\sigma,\tau \in \Omega(G)$, we say that $\sigma\prec_G \tau$ if $\sigma(v)\le \tau(v)$ for all vertices $v\in A$ and $\sigma(w)\ge \tau(w)$ for all vertices $w\in B$.

For a probability measure $\mu$ on $\Omega(G)$ and a function $f:\Omega(G)\to\mathbb{R}$, we let 
\[ \mu(f) = \sum_{\sigma\in\Omega(G)}f(\sigma)\mu(\sigma) \]
be the expectation of $f$ under $\mu$. 

Given two probability measures $\mu,\nu$ on $\Omega(G)$, we say that $\mu\prec_G\nu$ if $\mu(f)\le\nu(f)$ for any (bounded) function $f$ that is non-decreasing with respect to the partial order $\prec_G$ on $\Omega(G)$. 
\begin{lemma}[Monotonicity of boundary conditions on finite bipartite graphs]\label[lemma]{lem:boundarymono}
    Let $G=(A\cup B,E)$ be a finite bipartite graph, and for some fixed $\lambda>0$ write $\nu$ for $\mu_{G,\lambda}$ and $\prec$ for $\prec_G$. For a set $L$ contained in one bipartition class and an assignment $\tau\in\{0,1\}^L$, write $\nu^{L,\tau}$ for $\nu$ conditioned on the event $\sigma|_L=\tau$. Then
    \begin{enumerate}
        \item if $L\subset A$, then $\nu^{L,\mathbf{0}}\prec \nu^{L,\tau}\prec \nu^{L,\mathbf{1}}$, and
        \item if $L\subset B$, then $\nu^{L,\mathbf{1}}\prec \nu^{L,\tau}\prec \nu^{L,\mathbf{0}}$,
    \end{enumerate}
    where $\mathbf{0}$ and $\mathbf{1}$ denote the all-zero and all-one assignments, respectively.
\end{lemma}
\Cref{lem:boundarymono} is a consequence of the FKG inequality. A proof of the corresponding monotonicity statement appears in a paper of van den Berg and Steif~\cite[Lem.~3.1]{vS94}, and an alternative proof in a slightly more restricted context is given by Galvin, Martinelli, Ramanan, and Tetali~\cite{GMRT11}.

Now, suppose that $T$ is an infinite, locally finite tree rooted at $u$, and let $T_k$ be the tree obtained by truncating $T$ at depth $k$.
Write $V(T)=A\cup B$ for the bipartition with $u\in A$, and let $L(k)$ be the set of leaves at depth $k$ in $T_k$.
Then, by the monotonicity established in \Cref{lem:boundarymono}, the limits
\begin{align*}
    \mu^0 &= \lim_{k\to\infty}\nu^{L(2k),\mathbf{0}}_{T_{2k}}&
    &\text{and}&
    \mu^1 &= \lim_{k\to\infty}\nu^{L(2k),\mathbf{1}}_{T_{2k}}
\end{align*}
exist as local weak limits and are hard-core measures on $T$~\cite[Lem.~3.2]{vS94}. For any other hard-core measure $\mu$ at fugacity $\lambda$ on $T$, conditioning on the spins at depth $2k$ and applying \Cref{lem:boundarymono} gives
\[ \mu^0 \prec_T \mu \prec_T \mu^1. \]
This reduces the problem of determining uniqueness to the problem of determining when $\mu^0=\mu^1$. Indeed, $\mu^0=\mu^1$ is equivalent to uniqueness on $T$ (see e.g.~\cite{Spi75,vS94,Zac83,Zac85}).
The fact $\mu^0 \prec_T \mu \prec_T \mu^1$ above implies that the measures $\mu^0$ and $\mu^1$ are extremal in the sense that they minimize and maximize the marginal of the root of the tree, respectively\footnote{Thus, they are extremal in the conventional sense that they cannot be written as a nontrivial convex combination of hard-core measures on $T$, but we do not need this concept here.}.
For the ratio recursion on a tree, the all-vacant and all-occupied boundary conditions correspond to initializing the leaf ratios at $0$ and $\infty$, respectively.
This characterizes the approach to the problem of uniqueness and phase coexistence on trees: one must determine whether fixing the spins of vertices at depth $k$ in a rooted tree can influence the marginal of the root in the limit $k\to\infty$.
For each $d\ge 2$, the theorem below achieves this for the infinite $d$-ary tree through an understanding of the dynamical system $r_{n+1}=f_{d,\lambda}(r_n)$, which we think of as a one-dimensional system parametrized by $\lambda$.

\begin{theorem}\label[theorem]{thm:arytreedynamics}
    Let $f_{d,\lambda}:[0,\infty]\to[0,\infty]$ be the function with $f_{d,\lambda}(\infty)=0$ and $f_{d,\lambda}(r) = \lambda(1+r)^{-d}$. Then, for all integers $d\ge 2$, the following hold.
    \begin{enumerate}
        \item\label{itm:aryfixed} For all $\lambda>0$, $f_{d,\lambda}$ has exactly one fixed point $r^*$.
        \item\label{itm:aryunique} If $0<\lambda < \lambda_c(d+1)$, then $r^*$ is an attracting fixed point of $f_{d,\lambda}$ whose basin of attraction is $[0,\infty]$.
        \item\label{itm:arycoexist} If $\lambda > \lambda_c(d+1)$, then $r^*$ is a repelling fixed point and there exist $r^0 < r^* < r^1$ such that $\{r^0,r^1\}$ is an attracting periodic orbit of $f_{d,\lambda}$ whose basin of attraction is $[0,\infty]\setminus\{r^*\}$. More precisely, if $r\in [0,r^*)$, then $\lim_{n\to\infty}f_{d,\lambda}^{\circ 2n}(r)=r^0$, and if $r\in(r^*,\infty]$, then $\lim_{n\to\infty}f_{d,\lambda}^{\circ 2n}(r)=r^1$.
        \item\label{itm:arytreeorder} We have
        \[ \lim_{\lambda\downarrow \lambda_c(d+1)}|r^1-r^0| = 0. \]
    \end{enumerate}
\end{theorem}

The proof of \Cref{thm:arytreedynamics} is standard and provided one supplies the details of some basic calculus, an exposition can be found in, e.g.\ the work of Galvin, Martinelli, Ramanan and Tetali~\cite[Sec.~2]{GMRT11}. See also Barvinok's monograph~\cite[Sec.~6.3]{Bar16a}. We give a brief sketch as it will be useful to refer to differences when we study biregular trees below.

\begin{proof}[Sketch proof of \Cref{thm:arytreedynamics}]
    For brevity, let $f=f_{d,\lambda}$ and write $g=f\circ f$ for the second iterate of the dynamical system. 
    
    \Cref{itm:aryfixed} follows from the facts that $f(0)>0$ and $f$ is monotone decreasing. If we solve for critical points of the system, we see that there is exactly one, located at $(r_c,\lambda_c)$, where $r_c=1/(d-1)$ and $\lambda_c=\lambda_c(d+1)$.

    For \Cref{itm:aryunique}, one has to verify that $r^*=r^*(\lambda)$ is a continuous and increasing function of $\lambda$.
    The critical value $\lambda_c$ is precisely the point at which $f'(r^*)=-1$. Equivalently, $(r^*(\lambda_c),\lambda_c)$ is a critical point of the dynamical system (and it is in fact the only critical point). We write $r_c=r^*(\lambda_c(d+1))$.
    
    For $\lambda < \lambda_c$, the fixed point $r^*$ is attracting, whereas for $\lambda > \lambda_c$, $r^*$ is repelling.
    To check the basin of attraction of $r^*$ for $\lambda<\lambda_c$, observe that $g(r)>r$ for $r<r^*$ and $g(r)<r$ for $r>r^*$.
    It follows that when $\lambda<\lambda_c$, for any $r<r^*$ the sequence $g^{\circ n}(r)$ is monotone increasing and approaches $r^*$ as $n\to \infty$. Similarly, for any $r>r^*$ the sequence $g^{\circ n}(r)$ is monotone decreasing and approaches $r^*$ as $n\to \infty$.
    
    To prove \Cref{itm:arycoexist,itm:arytreeorder}, we observe that $g$ is an $S$-shaped function (see~\cite[Sec.~2.2]{GMRT11}) in the sense that
    \begin{itemize}
        \item $g$ is continuous on $[0,\infty]$ and differentiable on $(0,\infty)$,
        \item $g$ is increasing on $(0,\infty)$ with $g(0)>0$ and $\sup_x g(x) < \infty$,
        \item there exists $\overline{x}\in(0,\infty)$ such that the derivative $g'$ is monotone increasing in the interval $(0, \overline{x})$ and monotone decreasing in the interval $(\overline{x}, \infty)$.
    \end{itemize}
    It is easy to show (see e.g.~\cite[Sec.~6.3]{Bar16a}) that such a function has at most three fixed points in $[0,\infty)$.

    The local behavior at $(r^*,\lambda_c)$ is a \emph{pitchfork bifurcation} of the second iterate of the dynamical system $r_{n+1}=g(r_n)$.
    See~\cite[Sec.~3.8]{Bac22} for the details of a pitchfork bifurcation and~\cite[Thm.~3.7 and Rem.~3.9]{Bac22} for sufficient conditions; the remark applies here because $r^*(\lambda)$ is a smooth moving fixed-point branch, and the required regularity and derivative conditions can be verified directly. The $S$-shape and uniqueness of the critical point extend this local bifurcation to the global fixed-point picture (cf.~\cite[Sec.~2.2]{GMRT11}).
    In summary, the type of pitchfork bifurcation we have here means the system has one fixed point for $\lambda \le \lambda_c$, which is attracting for $\lambda<\lambda_c$. As $\lambda$ passes above $\lambda_c$, the smooth continuation of this fixed point switches to repelling, and two new attracting fixed points appear.
    This proves \Cref{itm:arytreeorder}. The bifurcation diagram is plotted in the case $d=2$ in \Cref{fig:arybif}.
    \Cref{itm:arycoexist} now follows from an analysis of the basin of attraction of these fixed points of $g$ analogous to the above argument for $f$.
\end{proof}

\begin{figure}[htb]
    \centering
    \includesvg[width=0.5\linewidth]{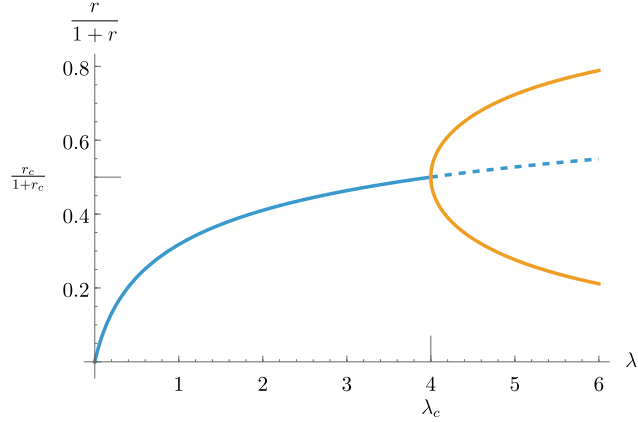}
    \caption{The bifurcation diagram for the dynamical system $r_{n+1} = f_{2,\lambda}\circ f_{2,\lambda}(r_n)$. 
    Solid lines indicate attracting fixed points and dashed lines indicate repelling fixed points.
    Note that the $r$ axis has been rescaled by the map $r\mapsto \frac{r}{1+r} \in [0,1]$ for convenience.}
    \label{fig:arybif}
\end{figure}

In light of \Cref{lem:boundarymono} and the subsequent discussion, \Cref{thm:regtree} follows easily from \Cref{thm:arytreedynamics}.

\begin{proof}[Sketch proof of \Cref{thm:regtree}]
Let $d=\Delta-1$ and let $u$ be an arbitrary vertex of $\TT_\Delta$. 
Note that $\TT_\Delta-u$ is a disjoint union of $\Delta$ copies of the infinite $d$-ary tree. 
Then, for any hard-core measure $\nu$ on $\TT_\Delta$, the marginal $\nu(u)$ lies between the extreme values given by the ratio--marginal coordinate transform $r\mapsto r/(1+r)$ applied to $f_{\Delta,\lambda}(r^1)$ and $f_{\Delta,\lambda}(r^0)$, where $r^0$ and $r^1$ from \Cref{thm:arytreedynamics} are obtained as limits of the root occupation ratios on the full $d$-ary tree of depth $2k$ under all-vacant and all-occupied boundary conditions, respectively.

For $\lambda<\lambda_c(\Delta)$, these limits must be equal and are the attracting fixed point $r^*$ given by \Cref{thm:arytreedynamics}\labelcref{itm:aryfixed,itm:aryunique}, whereas for $\lambda>\lambda_c(\Delta)$, they are precisely the $r^0$ and $r^1$ of \Cref{thm:arytreedynamics}\labelcref{itm:arycoexist}.
The phase transition at $\lambda_c(\Delta)$ is second-order because of \Cref{thm:arytreedynamics}\labelcref{itm:arytreeorder}, the monotonicity of \Cref{lem:boundarymono}, and the continuity of $f_{\Delta,\lambda}$.
\end{proof}

\section{Biregular trees}

For positive integers $\delta,\Delta$, let $\TT_{\delta,\Delta}$ be the \emph{infinite $(\delta,\Delta)$-biregular tree}. That is, $\TT_{\delta,\Delta}$ is the unique infinite tree in which each vertex has degree in the set $\{\delta,\Delta\}$, and all edges lie between a vertex of degree $\delta$ and a vertex of degree $\Delta$. In this section, we show that $\TT_{\delta,\Delta}$ undergoes a first-order phase transition when $\delta \ne \Delta$. Throughout this section, we let $a = \delta - 1$ and $b = \Delta - 1$.

\begin{theorem}\label[theorem]{thm:abtreedynamics}
    Fix $\lambda \ge 0$ and integers $2\le a<b$. Let $g:[0,\infty]\to[0,\infty]$ be the function $g = f_{a,\lambda}\circ f_{b,\lambda}$. There is a value $\overline{\lambda}$, defined in \Cref{lem:bifpoint}, such that the following statements hold.
    \begin{enumerate}
        \item\label{itm:abarytreeunique} If $\lambda < \overline{\lambda}$, then $g$ has exactly one fixed point $r^1$, which is attracting and has basin of attraction $[0,\infty]$.
        \item\label{itm:abarytreecrit} If $\lambda = \overline{\lambda}$, then $g$ has two fixed points $r^0 < r^1$ such that $g'(r^0)=1$ and $r^1$ is attracting.
        \item\label{itm:abarytreecoexist} If $\lambda > \overline{\lambda}$, then $g$ has three fixed points $r^0 < r^* < r^1$ such that $r^*$ is repelling and $r^0$ and $r^1$ are attracting.
        \item\label{itm:abarytreeorder} We have
        \[ \lim_{\lambda\downarrow\overline{\lambda}}|r^1-r^0| > 0. \]
        \item\label{itm:abflipped} The behavior of $h = f_{b,\lambda}\circ f_{a,\lambda}$ is similar and can be described in terms of the parameter above as follows:
        \begin{itemize}
            \item If $\lambda < \overline{\lambda}$, then $h$ has exactly one fixed point $f_{b,\lambda}(r^1)$, which is attracting.
            \item If $\lambda = \overline{\lambda}$, then $h$ has two fixed points $f_{b,\lambda}(r^1) < f_{b,\lambda}(r^0)$ such that $h'(f_{b,\lambda}(r^0))=1$ and $f_{b,\lambda}(r^1)$ is attracting.
            \item If $\lambda > \overline{\lambda}$, then $h$ has three fixed points $f_{b,\lambda}(r^1) < f_{b,\lambda}(r^*) < f_{b,\lambda}(r^0)$ such that $f_{b,\lambda}(r^*)$ is repelling and $f_{b,\lambda}(r^0)$ and $f_{b,\lambda}(r^1)$ are attracting.
        \end{itemize}
    \end{enumerate}
\end{theorem}

The fixed-point jump in \Cref{itm:abarytreeorder} is the source of the first-order transition proved in \Cref{thm:abfirstorder}. \Cref{fig:gversusrplot} illustrates the three cases in \Crefrange{itm:abarytreeunique}{itm:abarytreecoexist} when $a=2$ and $b=3$. The intersections of the curve with the line $g(r)=r$ give the fixed points in each case.

\begin{figure}[htb]
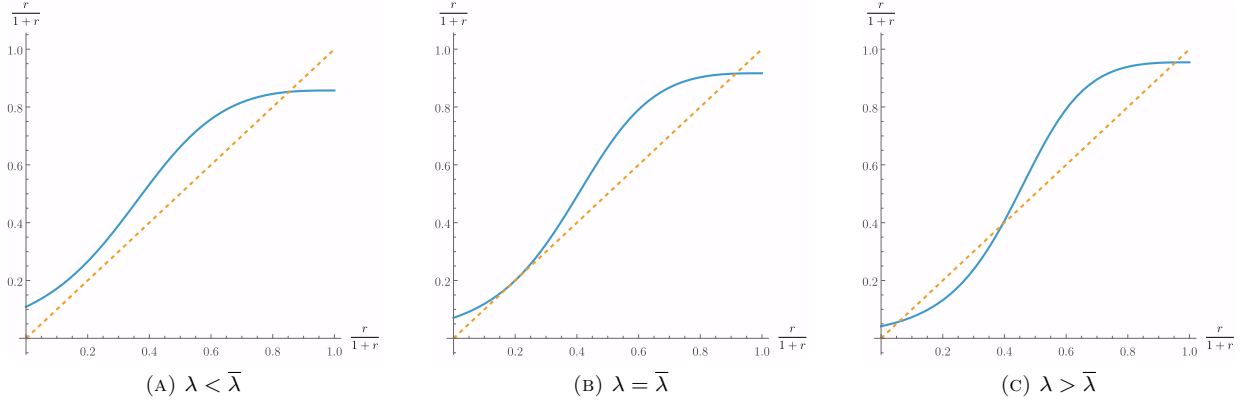

    \centering
    \subfloat[$\lambda < \overline{\lambda}$]{\includesvg[width=0.3\linewidth]{figures/subcrit.svg}}\qquad
    \subfloat[$\lambda = \overline{\lambda}$]{\includesvg[width=0.3\linewidth]{figures/crit.svg}}\qquad
    \subfloat[$\lambda > \overline{\lambda}$]{\includesvg[width=0.3\linewidth]{figures/supcrit.svg}}
    \caption{Plots of $g = f_{a,\lambda}\circ f_{b,\lambda}$ (solid blue line) and the line $g(r)=r$ (dashed orange line), where $a = 2$, $b = 3$, and $\overline{\lambda}\approx 11.002$. Note that both axes have been rescaled by the map $r\mapsto \frac{r}{1+r} \in [0,1]$ for convenience.}
    \label{fig:gversusrplot}
\end{figure}

\begin{corollary}\label[corollary]{thm:abfirstorder}
    The value $\overline{\lambda}$ in \Cref{thm:abtreedynamics} is a first-order phase transition for the infinite $(a+1,b+1)$-biregular tree $\TT_{a+1,b+1}$.
\end{corollary}
Unlike the infinite $d$-ary tree, the associated dynamical system is verified below to undergo a \emph{saddle-node} bifurcation~\cite[Thm.~3.5]{Bac22}. The sign calculation determines that the bifurcating pair lies on the side $\lambda>\overline{\lambda}$. We also use the following facts about $g$, which follow by direct differentiation. The key idea is that, by Rolle's theorem, a solution to $g'(r) = 1$ must lie between any pair of distinct fixed points of $g$.

\begin{enumerate}
    \item For $\lambda>0$, the function $g$ is strictly increasing and is strictly concave when $0<\lambda\le (b+1)/(ab-1)$. When $\lambda>(b+1)/(ab-1)$, it is $S$-shaped with unique inflection point
    \[r_{\mathrm{inf}}=\left(\frac{(ab-1)\lambda}{b+1}\right)^{1/b}-1.\]
    In all cases, $g(0) = \frac{\lambda}{(1+\lambda)^a}$ and $\sup_r g(r) = \lambda$.
    \item The function $g$ must have at least one fixed point and at most three fixed points for all $\lambda \in (0, \infty)$.
\end{enumerate}

Before we prove \Cref{thm:abtreedynamics,thm:abfirstorder}, we establish some properties of $g$ that will be of use.

\begin{lemma}\label[lemma]{lem:bifpoint} The map $g$ has a unique critical fixed point $(r_c, \lambda_c)$. It satisfies
    \[\frac{1}{ab-1}<r_c<\frac{1}{b-1}\]
    and is determined by the following equations:
     \begin{enumerate}
        \item $(1+r_c)^{b+1}(abr_c - (1+r_c))^{a-1} = (ab)^ar_c^{a+1}$, and
        \item $\lambda_c = \frac{(1+r_c)^{b+1}}{abr_c - (1+r_c)}$.
    \end{enumerate}
\end{lemma}
\begin{proof}
    A critical fixed point must satisfy the following conditions:
    \begin{axioms}
  \axitem[(FP)]\label{eq:fixedptbifpt} $g(r_c, \lambda_c) = r_c \iff (1+r_c)^{ab}\lambda_c = r_c(\lambda_c + (1+r_c)^b)^a$, and
  \axitem[(CR)]\label{eq:critbifpt} using \Cref{eq:fixedptbifpt}, $\frac{\partial g}{\partial r}|_{(r, \lambda) = (r_c, \lambda_c)} = 1 \iff abr_c\lambda_c = (1+r_c)^{b+1} + \lambda_c(1+r_c)$.
\end{axioms}
\Cref{eq:fixedptbifpt} ensures that $r_c$ is a fixed point, and \Cref{eq:critbifpt} ensures that it is critical. Rearranging \Cref{eq:critbifpt} gives
\[\lambda_c = \frac{(1+r_c)^{b+1}}{abr_c - (1+r_c)}.\]
Substituting this expression into \Cref{eq:fixedptbifpt} gives
\[(1+r_c)^{b+1}(abr_c - (1+r_c))^{a-1} = (ab)^ar_c^{a+1}.\]
Let $c=ab-1$ and define
\[Q(r)=\frac{(1+r)^{b+1}(cr-1)^{a-1}}{r^{a+1}}.\]
The equation above is $Q(r)=(ab)^a$, and a logarithmic derivative gives
\[\frac{Q'(r)}{Q(r)}=
\frac{c(b-1)\left(r-\frac{1}{b-1}\right)\left(r-\frac{a+1}{c}\right)}
{r(1+r)(cr-1)}.\]
Since $a<b$, we have
\[\frac1c<\frac1{b-1}<\frac{a+1}{c}.\]
The two extrema of $Q$ both lie above $(ab)^a$. Indeed, at the first one the ratio is
\[\frac{Q(1/(b-1))}{(ab)^a}
=\frac{b^b/(b-1)^{b-1}}{a^a/(a-1)^{a-1}}>1,\]
since $x^x/(x-1)^{x-1}$ is strictly increasing for $x>1$. At the second one, the logarithm of this ratio is
\[K(a,b)=\log\left(\frac{a^b(b+1)^{b+1}(ab-1)^{a-b}}{b^a(a+1)^{a+1}}\right).\]
Here $K(b,b)=0$, while the inequality $\log x\le x-1$ gives
\[\frac{\partial K}{\partial a}<\frac{a-b}{a(ab-1)}<0\]
for $1<a<b$, and hence $K(a,b)>0$. Finally, $Q(r)\to0$ as $r\downarrow1/c$ and $Q(r)\to\infty$ as $r\to\infty$. It follows that there is a unique solution, and that it lies in the interval $(1/(ab-1),1/(b-1))$. The formula for $\lambda_c$ then gives a unique positive value of $\lambda_c$.
\end{proof}
Equivalently, $r_c$ is the unique admissible root of a polynomial equation whose degree depends only on $a$ and $b$, and $\lambda_c$ is then given by \Cref{lem:bifpoint}.
\begin{lemma}\label[lemma]{lem:saddlenode}
    Let $(r_c, \lambda_c)$ be as defined in \Cref{lem:bifpoint}. Then $g$ undergoes a saddle-node bifurcation at $(r_c, \lambda_c)$.
\end{lemma}
\begin{proof}
    Let $(r_c, \lambda_c)$ be the critical fixed point defined in \Cref{lem:bifpoint}, and write $H(r,\lambda)=g(r,\lambda)-r$. Since $g$ is smooth near this positive critical point, the saddle-node criterion~\cite[Thm.~3.5]{Bac22} applies once, in addition to \Cref{eq:fixedptbifpt,eq:critbifpt}, the following conditions are verified:
\begin{enumerate}
    \item\label[condition]{conditionfirstder} $\frac{\partial g}{\partial\lambda}(r_c, \lambda_c) \ne 0$, and
    \item\label[condition]{conditionsecondder} $\frac{\partial^2 g}{\partial r^2}(r_c, \lambda_c) \ne 0$.
\end{enumerate}
We have
\[\frac{\partial g}{\partial \lambda}
=\frac{(1+r)^{ab}((1-a)\lambda+(1+r)^b)}{(\lambda+(1+r)^b)^{a+1}}\]
and
\[\frac{\partial^2g}{\partial r^2}
=\frac{ab\lambda^2(1+r)^{ab-2}}{(\lambda+(1+r)^b)^{a+2}}
\left((ab-1)\lambda-(b+1)(1+r)^b\right).\]
Substituting the formula for $\lambda_c$ from \Cref{lem:bifpoint} and using
\[r_c<\frac1{b-1}<\frac{a+1}{ab-1}\]
gives
\begin{align*}
\operatorname{sgn}\left(\frac{\partial g}{\partial\lambda}(r_c,\lambda_c)\right)
&=\operatorname{sgn}((b-1)r_c-1)<0,\\
\operatorname{sgn}\left(\frac{\partial^2g}{\partial r^2}(r_c,\lambda_c)\right)
&=\operatorname{sgn}(a+1-(ab-1)r_c)>0.
\end{align*}
Thus, \Cref{conditionfirstder,conditionsecondder} hold and the bifurcation is a saddle-node. Moreover, the equation $H(r,\lambda)=0$ locally defines $\lambda$ as a function of $r$, and implicit differentiation gives
\[\lambda'(r_c)=0,
\qquad
\lambda''(r_c)=-\frac{\frac{\partial^2}{\partial r^2}H(r_c,\lambda_c)}{\frac{\partial}{\partial\lambda}H(r_c,\lambda_c)}>0.\]
Consequently, the bifurcating pair occurs for $\lambda>\lambda_c$; its lower branch is attracting and its upper branch is repelling.
\end{proof}

\Cref{lem:saddlenode} shows that, as $\lambda$ increases through the bifurcation value, a pair of fixed points with opposite stability appears, as illustrated in \Cref{fig:biarybif}. By contrast, the dynamical system associated with the infinite $d$-ary tree undergoes a pitchfork bifurcation, as depicted in \Cref{fig:arybif}. We are now ready to prove \Cref{thm:abtreedynamics,thm:abfirstorder}.
\begin{proof}[Proof of \Cref{thm:abtreedynamics}]
Let $\overline{\lambda}=\lambda_c$ be the value from \Cref{lem:bifpoint}, and again write $H(r,\lambda)=g(r,\lambda)-r$, displaying the dependence of $g$ on $\lambda$. For $\lambda>0$, we have $H(0,\lambda)>0$ and $H(r,\lambda)\to-\infty$ as $r\to\infty$. Moreover, every fixed point satisfies $0<r<\lambda$, and $H(\,\cdot\,,\lambda)$ has at most three zeros. By \Cref{lem:bifpoint}, $(r_c,\overline{\lambda})$ is the only multiple fixed point over all $\lambda>0$. Since the fixed points remain bounded on bounded parameter intervals, their number can therefore change only at $\overline{\lambda}$. For sufficiently small positive $\lambda$, strict concavity gives exactly one fixed point. The oriented saddle-node in \Cref{lem:saddlenode} then gives exactly one fixed point for $0<\lambda<\overline{\lambda}$ and exactly three for $\lambda>\overline{\lambda}$. When $\lambda=0$, the map is identically zero and again has a unique fixed point.

All fixed points away from $\overline{\lambda}$ are simple. Since $H$ is positive at zero and negative at infinity, its derivative at the unique fixed point below $\overline{\lambda}$ is negative. Above $\overline{\lambda}$, the derivatives of $H$ at the three fixed points alternate in sign as negative, positive, and negative. Since $g'>0$, this proves the stability assertions in \Cref{itm:abarytreeunique,itm:abarytreecoexist}. If $0<\lambda<\overline{\lambda}$ and $r<r^1$, then $r<g(r)<r^1$, while if $r>r^1$, then $r^1<g(r)<r$. Thus, the iterates converge monotonically to $r^1$ from every initial value in $[0,\infty]$. For $\lambda=0$, convergence occurs in one step. This proves the basin assertion in \Cref{itm:abarytreeunique}.

\begin{figure}[htb]
    \centering
    \subfloat[$f_{a,\lambda}\circ f_{b,\lambda}$]{\includesvg[width=0.40\linewidth]{figures/2_3_FOPT.svg}}\qquad
    \subfloat[$f_{b,\lambda}\circ f_{a,\lambda}$]{\includesvg[width=0.40\linewidth]{figures/3_2_FOPT.svg}}
    \caption{The bifurcation diagram for the dynamical system $r_{n+1} = f_{a,\lambda}\circ f_{b,\lambda}(r_n)$, where $a = 2$, $b = 3$, and $\overline{\lambda} \approx 11.002$.
    Solid lines indicate attracting fixed points and dashed lines indicate repelling fixed points.
    Note that the $r$ axis has been rescaled by the map $r\mapsto \frac{r}{1+r} \in [0,1]$ for convenience.}
    \label{fig:biarybif}
\end{figure}

At $\lambda=\overline{\lambda}$, \Cref{lem:saddlenode} gives $\frac{\partial^2}{\partial r^2}H(r_c,\overline{\lambda})>0$, so $r_c$ is a strict local minimum of $H(\,\cdot\,,\overline{\lambda})$ at height zero. Since $H(r,\overline{\lambda})\to-\infty$, there is a fixed point $r^1>r_c$. Moreover, $H$ is positive both at zero and immediately to the left of $r_c$. Any fixed point below $r_c$ would therefore force at least two such fixed points, contradicting the at-most-three-fixed-points property after including $r_c$ and $r^1$. Thus, these are the only two fixed points. The fixed point $r^1$ is a downward crossing, and hence $0<g'(r^1)<1$. Setting $r^0=r_c$ proves \Cref{itm:abarytreecrit}.

As $\lambda\downarrow\overline{\lambda}$, the lower saddle-node branch satisfies $r^0(\lambda)\to r_c$. The larger fixed point at $\overline{\lambda}$ is simple, so the implicit function theorem gives $r^1(\lambda)\to r^1(\overline{\lambda})$. Consequently,
\[\lim_{\lambda\downarrow\overline{\lambda}}
|r^1(\lambda)-r^0(\lambda)|
=r^1(\overline{\lambda})-r_c>0,\]
which proves \Cref{itm:abarytreeorder}.

Finally, $r\mapsto f_{b,\lambda}(r)$ is an order-reversing bijection from the fixed points of $g$ to those of $h$, with inverse $s\mapsto f_{a,\lambda}(s)$. At corresponding fixed points,
\[h'(f_{b,\lambda}(r))=g'(r).\]
This proves all the claims in \Cref{itm:abflipped}.
\end{proof}

\begin{proof}[Proof of \Cref{thm:abfirstorder}]
    Let $u$ and $v$ be arbitrary vertices of degrees $a+1$ and $b+1$, respectively, in $\TT_{a+1,b+1}$. Then $\TT_{a+1,b+1}-u$ is the disjoint union of $a+1$ rooted copies of the $(b,a)$-ary tree $T_{b,a}$, in which vertices at even depth have $b$ children and vertices at odd depth have $a$ children. Similarly, $\TT_{a+1,b+1}-v$ is the disjoint union of $b+1$ copies of $T_{a,b}$.

    By \Cref{prop:rationbrs} (applied to finite truncations), the root occupation ratios of $T_{b,a}$ and $T_{a,b}$ can be found by iterating the dynamical systems studied in \Cref{thm:abtreedynamics}, given by the functions $h$ and $g$, respectively. To obtain the ratios for $u$ and $v$ in $\TT_{a+1,b+1}$, we apply the functions $f_{a+1,\lambda}$ and $f_{b+1,\lambda}$ to these root ratios from $T_{b,a}$ and $T_{a,b}$, respectively.

    At $\lambda=0$, the hard-core measure is concentrated on the empty independent set and is unique. For $0<\lambda<\overline{\lambda}$, the global basin assertion in \Cref{itm:abarytreeunique} makes the two extremal boundary limits coincide, so the monotonicity of \Cref{lem:boundarymono} gives uniqueness on $\TT_{a+1,b+1}$. For $\lambda>\overline{\lambda}$, iteration from the two extremal boundary conditions converges to the two outer fixed points, giving distinct extremal hard-core measures and hence phase coexistence.

    It remains to verify the order of the transition in the marginal coordinates of \Cref{def:orders}. Let $m(x)=x/(1+x)$ and take $v$ to have degree $b+1$. For brevity we suppress the dependence of $r^0$ and $r^1$ on $\lambda$. The extremal-measure sandwich following \Cref{lem:boundarymono} gives, for $\lambda>\overline{\lambda}$,
    \[D_{\TT_{a+1,b+1},v}(\lambda)
    =\left|
    m\bigl(f_{b+1,\lambda}(r^0)\bigr)
    -m\bigl(f_{b+1,\lambda}(r^1)\bigr)
    \right|.\]
    The corresponding formula using $h$ and $f_{a+1,\lambda}$ holds at vertices of degree $a+1$, so the one-sided limit in \Cref{def:orders} exists at every vertex. By the branch limits established in the proof of \Cref{thm:abtreedynamics},
    \begin{align*}
    \lim_{\lambda\downarrow\overline{\lambda}}
    D_{\TT_{a+1,b+1},v}(\lambda)
    &=\left|
    m\bigl(f_{b+1,\overline{\lambda}}(r_c)\bigr)
    -m\bigl(f_{b+1,\overline{\lambda}}(r^1)\bigr)
    \right|\\
    &>0,
    \end{align*}
    since $r_c<r^1$ and $f_{b+1,\overline{\lambda}}$ is strictly decreasing. Thus, $\overline{\lambda}$ is a first-order phase transition.
\end{proof}

\bibliographystyle{habbrv}
\bibliography{bib}

\end{document}